\documentclass{amsart}

\DeclareMathOperator{\charac}{char}
\DeclareMathOperator{\gr}{\operatorname{\mathsf{gr}}}

\newcommand{\hatalpha}{\widehat{\alpha}}
\renewcommand{\tilde}{\widetilde}
\newcommand{\grF}{\mathsf{F}}
\newcommand{\grV}{\mathsf{V}}
\newcommand{\grU}{\mathsf{U}}
\newcommand{\grW}{\mathsf{W}}
\newcommand{\Q}{\mathbb{Q}}
\newcommand{\F}{\mathbb{F}}
\newcommand{\Z}{\mathbb{Z}}

\newtheorem{prop}{Proposition}
\newtheorem{lem}[prop]{Lemma}
\newtheorem{thm}[prop]{Theorem}
\newtheorem{cor}[prop]{Corollary}

\theoremstyle{definition}
\newtheorem{example}[prop]{Example}
\newtheorem{definition}[prop]{Definition}

\title[Tame quadratic forms]{Springer's theorem for tame quadratic
  forms over Henselian fields}
\author{M. A. Elomary}
\address{Moulay Ismail University, F.S.T. Errachidia\\
B.P. 509, Boutalamine, Errachidia\\
Morocco}
\email{elomary@fste-umi.ac.ma}

\author{J.-P. Tignol}
\address{Universit\'e de Louvain\\
B~1348 Louvain-la-Neuve, Belgium}
\email{Jean-Pierre.Tignol@uclouvain.be}

\thanks{The first author would like to thank the second author and UCL
  for their hospitality while the work for this paper was done. He
  gratefully acknowledges support from the UCL Conseil des relations
  internationales. The
  second author is partially supported by the F.R.S.--FNRS (Belgium).}


\markleft{M.~A.~Elomary and J.-P.~Tignol}
\begin{document}

\begin{abstract}
  A quadratic form over a Henselian-valued field of arbitrary residue
  characteristic is tame if it becomes hyperbolic over a tamely
  ramified extension. The Witt group of tame quadratic forms is shown
  to be canonically isomorphic to the Witt group of graded quadratic
  forms over the graded ring associated to the filtration defined by
  the valuation, hence also isomorphic to a direct sum of copies of
  the Witt group of the residue field indexed by the value group
  modulo~2. 
\end{abstract}

\maketitle

A celebrated theorem of Springer \cite{S} establishes an isomorphism
between the Witt group of a complete discretely valued field and the
direct sum of two copies of the Witt group of the residue field,
provided the residue characteristic is different from~$2$. Springer
also considered the case where the residue characteristic is $2$, and
he pointed out the extra complications that arise from the fact that
the residue forms are not necessarily nonsingular, even when the
residue field is perfect. The exhaustive analysis by Aravire--Jacob
\cite{AJ} shows that the description of the Witt group of a dyadic
Henselian field is extremely delicate, even when the field is
maximally complete of characteristic~$2$.

Our purpose is to show that, by contrast, Springer's original
techniques---as revisited in \cite{RTW}---yield a very general
characteristic-free version of Springer's theorem for the tame part of
the Witt group. Our main result is the following: let $F$ be a field
with a Henselian valuation $v$, with value group $\Gamma_F$ and
residue field $\overline{F}$ of arbitrary characteristic; there is a
group isomorphism
\[
\partial\colon I_{qt}(F) \stackrel{\sim}{\to}
\bigoplus_{\Gamma_F/2\Gamma_F} I_q(\overline{F})
\]
where $I_q(\overline{F})$ is the quadratic Witt group of
$\overline{F}$, consisting of Witt-equivalence classes of
even-dimensional nonsingular quadratic forms over $\overline{F}$, and
$I_{qt}(F)\subseteq I_q(F)$ is the subgroup of Witt classes that are
split by a tamely ramified extension of 
$F$. The isomorphism $\partial$ is defined by residue forms, which
depend on the choice of ``uniformizing parameters'' $\pi_\delta\in
F^\times$ for representatives $\delta$ of the various cosets of
$\Gamma_F$ modulo $2\Gamma_F$. It is actually obtained as the
composition of a \emph{canonical} isomorphism
\begin{equation}
  \label{eq:1}
  I_{qt}(F)\stackrel{\sim}{\to} I_q\bigl(\gr(F)\bigr),
\end{equation}
where $\gr(F)$ is the graded ring associated with the filtration of
$F$ defined by the valuation $v$, and an isomorphism
\begin{equation}
  \label{eq:2}
  I_q\bigl(\gr(F)\bigr)\stackrel{\sim}{\to}
  \bigoplus_{\Gamma_F/2\Gamma_F} I_q(\overline{F})
\end{equation}
depending on the choice of uniformizing parameters. If $\charac
\overline{F}\neq2$, the isomorphism $\partial$ extends to the whole
Witt group $W(F)$, as is well-known (see \cite[Satz~3.1]{Tietze} of
\cite{RTW}, for instance).
The case where $\charac\overline{F}=2$ was also considered by Tietze
in \cite{Tietze}\footnote{We are grateful to S.~Garibaldi for calling
  our attention on this reference.}, who also obtained an isomorphism
from a subgroup $U(F)\subseteq I_q(F)$ onto
$\bigoplus_{\Gamma_F/2\Gamma_F} I_q(\overline{F})$. Tietze's method is
quite different: it uses a presentation of Witt groups by generators
and relations, and $U(F)$ is defined in \cite{Tietze} as the subgroup
generated by the Witt classes of quadratic forms $ax^2+bxy+cy^2$ with
$v(a)=v(b)=v(c)$. Comparing Tietze's results with ours, it appears
that actually $U(F)=I_{qt}(F)$, see Corollary~\ref{cor:inertsplit}.

Sections~\ref{sec:gradedWitt} and \ref{sec:norms} introduce the
formalism of norms on vector spaces and graded rings. The Witt group
of graded fields is discussed in \S\ref{sec:gradedWitt}, where we
prove the isomorphism~\eqref{eq:2}: see
Theorem~\ref{thm:gradedWitt}. The isomorphism \eqref{eq:1} is obtained
in \S\ref{sec:norms} (see Corollary~\ref{cor:deliso}), where
$I_{qt}(F)$ is defined as the group of Witt classes of quadratic
spaces that admit a certain type of norms, which we call \emph{tame
  norms}. In \S\ref{sec:tame}, we establish the alternative
characterization of $I_{qt}(F)$ in terms of tame splitting fields.

The following notation will be used consistently throughout: $F$ is a
field with a (Krull) valuation $v\colon F\to\Gamma\cup\{\infty\}$,
where $\Gamma$ is a totally ordered group. Without loss of generality,
we assume $\Gamma$ is divisible, since we may substitute for $\Gamma$
its divisible hull. We let $\Gamma_F$ denote the value group of $F$,
i.e., $\Gamma_F=v(F^\times)\subseteq\Gamma$, and let $\overline{F}$ denote
the residue field of $F$.

For quadratic forms, we generally follow the terminology of
\cite{EKM}: if $q\colon V\to F$ is a quadratic form on an $F$-vector
space $V$, we let $b\colon V\times V\to F$ denote the polar form
\[
b(x,y)=q(x+y)-q(x)-q(y)\qquad\text{for $x$, $y\in V$.}
\]
The form $b$ is \emph{nondegenerate} if $x=0$ is the only vector such
that $b(x,y)=0$ for all $y\in V$. We call the quadratic form $q$
\emph{regular} if $x=0$ is the only vector in $V$ such that $q(x)=0$
and $b(x,y)=0$ for all $y\in V$, and \emph{nondegenerate} if it
remains regular after scalar extension to an algebraic closure of
$F$. We say $q$ is \emph{nonsingular} 
if its polar form is nondegenerate; thus, if $\charac F\neq2$ or if
$\dim q$ is even, $q$ is nonsingular if and only if it is
nondegenerate. If $\charac F=2$, all
nonsingular forms are even-dimensional since their polar form is
nondegenerate and alternating, but $x_1x_2+x_3^2$ is an
odd-dimensional nondegenerate form.

\section{Quadratic Witt group of graded fields}
\label{sec:gradedWitt}
Let $\Gamma$ be a divisible torsion-free abelian group, which will
contain the degrees of all the graded structures we shall
consider. Graded commutative rings in which every nonzero homogeneous
element is invertible are called \emph{graded fields}, and graded
modules over graded fields are called \emph{graded vector
  spaces}. Since $\Gamma$ is torsion-free, graded fields are domains
and graded vector spaces are free modules, see \cite[\S1]{HW}. The
rank of a graded vector space is called its \emph{dimension}. Let
\[
\grF=\bigoplus_{\gamma\in\Gamma}\grF_\gamma
\]
be a graded field and
\[
\grV=\bigoplus_{\gamma\in\Gamma}\grV_\gamma
\]
be a graded $\grF$-vector space. We let $\Gamma_\grF$, $\Gamma_\grV$ denote
the sets of degrees of $\grF$ and $\grV$, i.e.,
\[
\Gamma_\grF=\bigl\{\gamma\in\Gamma\mid \grF_\gamma\neq\{0\}\bigr\},
\qquad
\Gamma_\grV=\bigl\{\gamma\in\Gamma\mid \grV_\gamma\neq\{0\}\bigr\}.
\]
The set $\Gamma_\grF$ is a subgroup of $\Gamma$, and $\Gamma_\grV$ is
a union of cosets of $\Gamma_\grF$. For each coset
$\Lambda\in\Gamma/\Gamma_F$, let
\[
\grV_\Lambda=\bigoplus_{\lambda\in\Lambda}\grV_\lambda
\]
(so $\grV_\Lambda=\{0\}$ if $\Lambda\not\subset\Gamma_V$). If
$\Gamma_\grV$ is the disjoint union of cosets $\Lambda_1$, \ldots,
$\Lambda_n\in\Gamma/\Gamma_F$, i.e.,
$\Gamma_\grV=\Lambda_1\sqcup\ldots\sqcup \Lambda_n$, we have a canonical
decomposition of $\grV$ into graded sub-vector spaces
\begin{equation}
  \label{eq:candec}
  \grV=\bigoplus_{i=1}^n\grV_{\Lambda_i}.
\end{equation}
Note that the homogeneous component $\grF_0$ of $\grF$ is a field, and
each $\grV_\gamma$ for $\gamma\in\Gamma$ is an $\grF_0$-vector
space. Picking an element $\lambda_i\in\Lambda_i$ for each $i=1$,
\ldots, $n$, we have
\[
\dim_\grF\grV=\sum_{i=1}^n\dim_\grF\grV_{\Lambda_i} =
\sum_{i=1}^n\dim_{\grF_0} \grV_{\lambda_i}. 
\]

Let $\grV$ be a finite-dimensional graded vector space over a graded
field $\grF$. A \emph{graded quadratic form} on $\grV$ is a map
\[
q\colon \grV\to\grF
\]
satisfying the following conditions involving $q$ and its polar form
$b\colon \grV\times\grV\to\grF$ defined by
\[
b(x,y)=q(x+y)-q(x)-q(y):
\]
\begin{enumerate}
\item[(i)]
$q(xa)=q(x)a^2$ for all $x\in\grV$, $a\in\grF$;
\item[(ii)]
$b$ is an $\grF$-bilinear form on $\grV$;
\item[(iii)]
$q(\grV_\gamma)\subseteq \grF_{2\gamma}$ for all $\gamma\in\Gamma$;
\item[(iv)]
$b(\grV_\gamma,\grV_\delta)\subseteq \grF_{\gamma+\delta}$ for all
$\gamma$, $\delta\in\Gamma$.
\end{enumerate}
These conditions imply in particular that
\begin{equation}
  \label{eq:conseq}
  q(\grV_\gamma)=\{0\}\text{ if
    $\gamma\notin\textstyle\frac12\Gamma_\grF$} \quad\text{and}\quad
  b(\grV_\gamma,\grV_\delta)=\{0\}\text{ if $\gamma+\delta\notin\Gamma_\grF$}.
\end{equation}
The graded quadratic form $q$ is called \emph{nonsingular} if its
polar form $b$ is nondegenerate, i.e., $x=0$ is the only vector such
that $b(x,y)=0$ for all $y\in V$. It is called \emph{hyperbolic} if it
is nonsingular and $\grV=\grU\oplus\grW$ for some graded subspaces
$\grU$, $\grW$ such that $q(\grU)=q(\grW)=\{0\}$.

\begin{prop}
  Let $q$ be a nonsingular graded quadratic form on $\grV$. The
  canonical decomposition \eqref{eq:candec} yields an orthogonal
  decomposition 
  \[
  \grV=\Bigl(\bigoplus_{\Lambda\in\frac12\Gamma_\grF/\Gamma_\grF}^\perp \grV_\Lambda\Bigr)
  \stackrel{\perp}{\oplus} \grW
  \qquad\text{where}\quad
  \grW=\bigoplus_{\Lambda\notin\frac12\Gamma_\grF/\Gamma_\grF} \grV_\Lambda.
  \]
  For each $\Lambda\in \frac12\Gamma_\grF/\Gamma_\grF$, the restriction
  of $q$ to $\grV_\Lambda$ is 
  nonsingular. Moreover, the restriction of $q$ to $\grW$ is hyperbolic.
\end{prop}

\begin{proof}
  All assertions except the last one really concern the polar form
  $b$; their easy proof (based on the observation \eqref{eq:conseq})
  can be found in \cite[Proposition~1.1]{RTW}. It is also proven there
  that the restriction of $b$ to $\grW$ is
  hyperbolic. Since moreover $q(\grV_\Lambda)=\{0\}$ for all cosets
  $\Lambda\not\subset\frac12\Gamma_\grF$ by 
  \eqref{eq:conseq}, the proposition follows.
\end{proof}

For each coset $\Delta\in\Gamma_\grF/2\Gamma_\grF$, fix an element
$\delta\in\Delta$ and a nonzero homogeneous element $\pi_\delta\in
\grF_{\delta}$ ($\delta=0$ and $\pi_0=1$ for $\Delta=2\Gamma_\grF$) and let
\begin{equation}
  \label{eq:resa}
  q_\Delta\colon\grV_{\frac12\delta}\to\grF_0,\qquad x\mapsto
  \pi_\delta^{-1}q(x). 
\end{equation}
Since the restriction of $q$ to $\grV_{\frac12\Delta}$ is nonsingular,
it follows that $q_\Delta$ is a nonsingular quadratic form on the
$\grF_0$-vector space $\grV_{\frac12\delta}$. Of course, this
quadratic form depends on the choice of $\pi_\delta$, except for
$\Delta=2\Gamma_\grF$. 
\medbreak
\par
Mimicking the usual construction, we may define a Witt equivalence of
nonsingular graded quadratic forms over a given graded field $\grF$
and endow the set of Witt-equivalence classes with a group structure
using the orthogonal sum. We let $I_q(\grF)$ denote the quadratic Witt
group of $\grF$, consisting of Witt-equivalence classes of
even-dimensional nonsingular graded quadratic forms over $\grF$.

\begin{thm}
  \label{thm:gradedWitt}
  The map that carries each nonsingular graded quadratic form $q$ to
  the collection of nonsingular $\grF_0$-quadratic forms
  $(q_\Delta)_{\Delta\in 
    \Gamma_\grF/2\Gamma_\grF}$ defines a group isomorphism:
  \[
  I_q(\grF)\stackrel{\sim}{\to} \bigoplus_{\Gamma_\grF/2\Gamma_\grF}
  I_q(\grF_0). 
  \]
  This isomorphism depends on the choice of the homogeneous elements
  $\pi_\delta\in \grF_{\delta}$.
\end{thm}

The proof is routine. See \cite[Proposition~1.5(iv)]{RTW} for the
(more complicated) case of even hermitian forms over graded division
algebras with involution.

\section{Norms and residues}
\label{sec:norms}
Let $(F,v)$ be an arbitrary valued field and $V$ a finite-dimensional
$F$-vector space. We recall from \cite{RTW} and \cite{TWgr} that a
\emph{$v$-value function} on $V$ is a map
\[
\alpha\colon V\to\Gamma\cup\{\infty\}
\]
satisfying the following properties, for $x$, $y\in V$ and $\lambda\in
F$:
\begin{enumerate}
\item[(i)] $\alpha(x)=\infty$ if and only if $x=0$;
\item[(ii)] $\alpha(x\lambda)=\alpha(x)+v(\lambda)$;
\item[(iii)] $\alpha(x+y)\geq \min\bigl(\alpha(x),\alpha(y)\bigr)$.
\end{enumerate}
The $v$-value function $\alpha$ is called a \emph{$v$-norm} if there
is a base $(e_i)_{i=1}^n$ of $V$ that splits $\alpha$ in the following
sense:
\[
\alpha\Bigl(\sum_{i=1}^ne_i\lambda_i\Bigr) =
\min\bigl(\alpha(e_i\lambda_i) \mid i=1,\ldots,n\bigr) \qquad\text{for
  $\lambda_1$, \ldots, $\lambda_n\in F$.}
\]
Any $v$-value function $\alpha$ on $V$ defines a filtration of $V$ by
modules over the valuation ring of $F$: for any $\gamma\in\Gamma$ we
let
\[
V^{\geq\gamma}=\{x\in V\mid \alpha(x)\geq\gamma\},\quad
V^{>\gamma} = \{x\in V\mid \alpha(x)>\gamma\},\quad
V_\gamma=V^{\geq\gamma}/V^{>\gamma},
\]
and we define
\[
\gr_\alpha(V)=\bigoplus_{\gamma\in\Gamma} V_\gamma.
\]
Similarly, let $\gr(F)$ be the graded ring associated to the
filtration of $F$ defined by the valuation. Since every nonzero
homogeneous element of $\gr(F)$ is invertible, this ring is a
graded field. The $F$-vector space structure on $V$ induces on
$\gr_\alpha(V)$ a structure of graded $\gr(F)$-module, so
$\gr_\alpha(V)$ is a graded $\gr(F)$-vector space. For any nonzero
$x\in V$ we let
\[
\tilde x = x+ V^{>\alpha(x)} \in V_{\alpha(x)}\subseteq \gr_\alpha(V).
\]
We also let $\tilde 0 = 0$ and use a similar notation for elements in
$\gr(F)$. It is shown in \cite[Corollary~2.3]{RTW} that a base
$(e_i)_{i=1}^n$ of $V$ splits $\alpha$ if and only if $(\tilde
e_i)_{i=1}^n$ is a $\gr(F)$-base of $\gr_\alpha(V)$, and that $\alpha$
is a norm if and only if
\[
\dim_{\gr(F)}\gr_\alpha(V) = \dim_FV.
\]

Now, let $q\colon V\to F$ be an arbitrary quadratic form, with polar
form $b\colon V\times V\to F$. We say that a $v$-value function
$\alpha$ is \emph{bounded by $q$}, and write $\alpha\prec q$, if the
following two conditions hold:
\begin{enumerate}
\item[(a)]
$v\bigl(b(x,y)\bigr) \geq \alpha(x)+\alpha(y)$ for all $x$, $y\in V$,
and
\item[(b)]
$v\bigl(q(x)\bigr) \geq 2\alpha(x)$ for all $x\in V$.
\end{enumerate}
Of course, letting $x=y$ in (a) yields (b) if $\charac F\neq2$. On the
other hand, (b) does not imply (a): see Example~\ref{ex:notbounded}
below. Note that Springer in \cite{S2} requires only~(b), whereas
Goldman--Iwahori in \cite{GI} require only~(a). Both conditions are
required by Bruhat--Tits in \cite[D\'efinition~2.1]{BT}.

For each $\gamma\in\Gamma$, let $p_\gamma\colon F^{\geq\gamma}\to
F_\gamma$ be the canonical map. When $\alpha\prec q$ we define maps
\[
\tilde q_\alpha\colon V_\gamma\to F_{2\gamma}\quad\text{and}\quad
\tilde b_\alpha\colon V_\gamma\times V_\delta\to F_{\gamma+\delta}
\quad\text{for $\gamma$, $\delta\in\Gamma$}
\]
by
\[
\tilde q_\alpha(\tilde x) = p_{2\gamma}\bigl(q(x)\bigr)
\quad\text{and}\quad \tilde b_\alpha(\tilde x, \tilde y) =
p_{\gamma+\delta}\bigl(b(x,y)\bigr)
\]
for $x$, $y\in V$ with $\alpha(x)=\gamma$ and $\alpha(y)=\delta$.
We extend $\tilde b_\alpha$ to $\gr_\alpha(V)$ by bilinearity and
define
\[
\tilde q_\alpha\colon\gr_\alpha(V)\to\gr(F)
\]
as follows: for $\tilde x_\gamma\in V_\gamma$,
\[
\tilde q_\alpha\Bigl(\sum_{\gamma\in\Gamma} \tilde x_\gamma\Bigr) =
\sum_\gamma \tilde q_\alpha(\tilde x_\gamma) + \sum_{\gamma<\delta}
\tilde b_\alpha(\tilde x_\gamma,\tilde x_\delta).
\]
The map $\tilde q_\alpha$ is a quadratic form on $\gr_\alpha(V)$ with
polar form $\tilde b_\alpha$. It is a graded quadratic
form since $\tilde q_\alpha(V_\gamma)\subseteq F_{2\gamma}$ and
$\tilde b_\alpha(V_\gamma,V_\delta)\subseteq F_{\gamma+\delta}$ for
$\gamma$, $\delta\in \Gamma$. The straightforward verifications are
omitted.

\begin{example}
  \label{ex:notbounded}
  Suppose $v$ is a discrete valuation on $F$ with
  $\Gamma_F=\mathbb{Z}$ and let $V=F^{\oplus2}$ with the hyperbolic
  quadratic form $q(x_1,x_2)=x_1x_2$. Define a norm $\alpha\colon
  V\to\frac12 \mathbb{Z}\cup\{\infty\}$ by
  \[
  \alpha(x_1,x_2) = \min(v(x_1), v(x_2)+{\textstyle\frac12})\qquad
  \text{for $x_1$, $x_2\in F$.}
  \]
  Clearly, we have $v\bigl(q(x)\bigr)\geq2\alpha(x)$ for all $x\in V$;
  but $b\bigl((1,0),(0,1)\bigr) = 1$, so
  \[
  v\bigl(b\bigl((1,0),(0,1)\bigr)\bigr) = 0 < \alpha(1,0) +
  \alpha(0,1) = \textstyle\frac12.
  \]
  Thus, condition~(b) holds but not (a). Also, note that for $x=(1,0)$
  and $x'=(1,1)$ we have $\tilde x = \tilde x'$ in $\gr_\alpha(V)$,
  but $q(x)=0$ and $q(x')=1$, so $\tilde q_\alpha(\tilde x)$ is not
  well-defined. 

  By contrast, the map $\beta\colon V\to\mathbb{Z}\cup\{\infty\}$
  defined by
  \[
  \beta(x_1,x_2)=\min\bigl(v(x_1), v(x_2)\bigr)\qquad \text{for $x_1$,
    $x_2\in F$}
  \]
  is bounded by $q$. The induced quadratic form $\tilde q_\beta$ is
  hyperbolic. 
\end{example}

\begin{example}
  \label{ex:quadext}
  Let $(F,v)$ be an arbitrary valued field and let $K/F$ be an
  arbitrary quadratic extension. Let $N\colon K\to F$ be the norm
  form, which is a quadratic form. We consider two cases:

  \noindent\emph{Case 1: The valuation $v$ extends to two different
    valuations $v_1$, $v_2$ on $K$.}

  Then $K/F$ is not purely inseparable, hence it is a Galois
  extension. Let $\sigma\colon K\to K$ be the nontrivial automorphism
  of $K/F$, so $v_2=v_1\circ\sigma$. We have $v\bigl(N(x)\bigr) =
  v_1(x)+v_2(x)$ for all $x\in V$. The map $\alpha\colon
  K\to\Gamma\cup\{\infty\}$ defined by
  \[
  \alpha(x)=\min\bigl(v_1(x),v_2(x)\bigr) \qquad\text{for $x\in K$}
  \]
  is a $v$-norm on $K$ by \cite[Corollary~1.7]{TWgr}, and it is
  readily checked that 
  $\alpha\prec N$. Let $u\in K$ be such that $v_1(u)\neq v_2(u)$. Then
  $\sigma(u)\neq -u$, hence after scaling by a suitable factor in
  $F^\times$ we may assume $u+\sigma(u)=1$. Since
  $v_1\bigl(\sigma(u)\bigr) = v_2(u)\neq v_1(u)$, this equation yields
  \[
  0 = \min\bigl(v_1(u), v_1\bigl(\sigma(u)\bigr)\bigr) = \alpha(u).
  \]
  It follows that $\tilde b_\alpha(\tilde 1,\tilde u)=\tilde 1$, so
  the form $\tilde N_\alpha$ is nonsingular. Moreover, $\tilde
  N_\alpha(\tilde u) =0$, so $\tilde N_\alpha$ is hyperbolic.

  \noindent\emph{Case 2: The valuation $v$ has a unique extension to
    $K$.}

  We again write $v$ for the valuation on $K$ extending $v$. This
  valuation is a $v$-value function on $K$ and it is readily checked
  that $v\prec N$; in fact $v\bigl(N(x)\bigr) = 2v(x)$ for all $x\in
  K$. If $K/F$ is immediate (i.e., $\overline{K}=\overline{F}$ and
  $\Gamma_K=\Gamma_F$), then $\gr_v(K)=\gr(F)$, so $v$ is not a
  norm. Otherwise, $\gr_v(K)$ is a quadratic extension of $\gr(F)$,
  and $\tilde N_\alpha$ is the norm form of that extension. It is
  nonsingular if and only if $K/F$ is tame, i.e., either $K/F$ is
  totally ramified and $\charac\overline{F}\neq2$, or
  $\overline{K}/\overline{F}$ is a separable quadratic extension
  (i.e., $K/F$ is inertial).
\end{example}

This last example suggests the following definition:

\begin{definition}
  Let $(F,v)$ be an arbitrary valued field and $(V,q)$ be a quadratic
  space over $F$. A $v$-norm $\alpha$ on $V$ is called a \emph{tame norm
    compatible with $q$} if $\alpha\prec q$ and the induced quadratic
  form $\tilde q_\alpha$ on $\gr_\alpha(V)$ is nonsingular.
\end{definition}

For any quadratic space $(V,q)$ and any $v$-value function $\alpha$ on
$V$ such that $\alpha\prec q$, it is clear that $\tilde
b_\alpha(\tilde x, \eta)=0$ for all $\eta\in \gr_\alpha(V)$ if
$b(x,y)=0$ for all $y\in V$. Therefore, tame compatible norms exist
only for nonsingular forms. If $\charac\overline{F}\neq2$, the tame
norms compatible with a quadratic form are exactly the norms that are
compatible with its polar form, in the terminology of \cite{RTW}; see
\cite[Remark~3.2]{RTW}. In that case, for every nonsingular quadratic
form there is a tame compatible norm, see
\cite[Corollary~3.6]{RTW}. If $\charac\overline{F}=2$, the tame norms
compatible with a quadratic form $q$ are those that are compatible
with its polar form and that moreover satisfy condition~(b):
$v\bigl(q(x)\bigr)\geq2\alpha(x)$ for all $x\in V$. There are
nonsingular forms for which there is no tame compatible norm, for
instance the norm forms of totally ramified quadratic extensions of
Henselian dyadic fields: see Theorem~\ref{thm:main}.

\begin{lem}
  \label{lem:sum}
  Let $(V_1,q_1)$ and $(V_2,q_2)$ be quadratic spaces over $F$, and
  let $\alpha_1$, $\alpha_2$ be tame $v$-norms on $V_1$, $V_2$ that
  are compatible with $q_1$ and $q_2$ respectively. Define
  $\alpha_1\oplus\alpha_2\colon V_1\oplus V_2\to \Gamma\cup\{\infty\}$
  by
  \[
  (\alpha_1\oplus\alpha_2)(x_1,x_2) = \min\bigl(\alpha_1(x_1),
  \alpha_2(x_2)\bigr)\qquad \text{for $x_1\in V_1$ and $x_2\in V_2$.}
  \]
  Then $\alpha_1\oplus\alpha_2$ is a tame $v$-norm on $V_1\oplus V_2$
  compatible with $q_1\perp q_2$, and there is a canonical
  identification of graded quadratic spaces
  \[
  (\gr_{\alpha_1\oplus\alpha_2}(V_1\oplus V_2), \tilde{(q_1\perp
    q_2)}_{\alpha_1\oplus\alpha_2}) =
  (\gr_{\alpha_1}(V_1),\tilde{q_1}_{\alpha_1}) \perp
  (\gr_{\alpha_2}(V_2), \tilde{q_2}_{\alpha_2}).
  \]
\end{lem}

The easy proof is omitted; see \cite[Example~3.7(iii)]{RTW}.

\begin{lem}
  \label{lem:hyperb}
  Let $(V,q)$ be a quadratic space over $F$ and let $\alpha$ be a tame
  $v$-norm on $V$ compatible with $q$. The Witt indices of $q$ and
  $\tilde q_\alpha$ are related by
  \[
  \mathfrak{i}_0(q)\leq \mathfrak{i}_0(\tilde q_\alpha).
  \]
  In particular, if $q$ is hyperbolic, then
  $\tilde q_\alpha$ is hyperbolic.
\end{lem}

\begin{proof}
  If $U\subseteq V$ is a totally $q$-isotropic subspace, then
  $\gr_\alpha(U)\subseteq \gr_\alpha(V)$ is a 
  totally $\tilde q_\alpha$-isotropic subspace with $\dim\gr_\alpha(U)
  = \dim U$.
\end{proof}

Let $I_q(F)$ denote the quadratic Witt group of $F$, consisting of
Witt-equivalence classes of even-dimensional nonsingular quadratic forms;
see \cite[\S8.B]{EKM}. From Lemma~\ref{lem:sum}, it follows that the
Witt-equivalence classes of even-dimensional nonsingular quadratic forms
that admit a tame compatible norm form a subgroup of $I_q(F)$. We let
$I_{qt}(F)$ denote this subgroup, which we call the \emph{tame
  quadratic Witt group of $F$}.

\begin{prop}
  \label{prop:residue}
  There is a well-defined group epimorphism
  \[
  \partial\colon I_{qt}(F) \to I_q\bigl(\gr(F)\bigr)
  \]
  that carries the Witt class of any even-dimensional nonsingular
  quadratic form 
  $q$ with a tame compatible norm $\alpha$ to the
  Witt class of $\tilde q_\alpha$.
\end{prop}

\begin{proof}
  If $\charac\overline{F}\neq2$, this is shown in
  \cite[Theorem~3.11]{RTW} (and holds for odd-dimensional nondegenerate
  quadratic forms as well). The arguments also hold if
  $\charac\overline{F}=2$: to see that the Witt equivalence class of
  $\tilde q_\alpha$ does not depend on the choice of the tame
  compatible norm $\alpha$, suppose $\beta$ is another tame compatible
  norm. Then $\alpha\oplus\beta$ is a tame norm compatible with the
  hyperbolic form $q\perp-q$, hence Lemmas~\ref{lem:sum} and
  \ref{lem:hyperb} show that $\tilde q_\alpha\perp-\tilde q_\beta$ is
  hyperbolic. 

  To prove surjectivity of $\partial$ when $\charac\overline{F}=2$, it
  suffices to show that the 
  Witt class of every nonsingular binary form over $\gr(F)$ is in the
  image. Let $\grV$ be a $2$-dimensional graded vector space over
  $\gr(F)$, and let $\varphi\colon\grV\to\gr(F)$ be a nonsingular
  quadratic form. Since the corresponding polar form $b_\varphi$ is
  nonsingular, we may find homogeneous vectors $\varepsilon_1$,
  $\varepsilon_2\in\grV$ such that
  $b_\varphi(\varepsilon_1,\varepsilon_2)=1$. These vectors form a
  base of $\grV$. If $\varphi(\varepsilon_1)=0$ or
  $\varphi(\varepsilon_2)=0$, then $\varphi$ is hyperbolic hence its
  Witt class is zero, which lies in the image of $\partial$. If
  $\varphi(\varepsilon_1)$, $\varphi(\varepsilon_2)$ are nonzero, then
  we may find $a_1$, $a_2\in F^\times$ such that
  $\varphi(\varepsilon_1)=\tilde a_1$ and
  $\varphi(\varepsilon_2)=\tilde a_2$. Note that the condition
  $b_\varphi(\varepsilon_1,\varepsilon_2)=1$ implies that
  $\deg\varepsilon_1+\deg\varepsilon_2=0$, hence
  $v(a_1)+v(a_2)=0$. Consider a $2$-dimensional $F$-vector space $U$
  with base $e_1$, $e_2$ and quadratic form
  \[
  q(e_1x_1+e_2x_2)=a_1x_1^2+x_1x_2+a_2x_2^2\qquad\text{for $x_1$,
    $x_2\in F$.}
  \]
  Straightforward computations show that $q$ is a nonsingular
  quadratic form and that the map $\alpha\colon
  U\to\Gamma\cup\{\infty\}$ defined by
  \[
  \alpha(e_1x_1+e_2x_2)=\min\bigl(\textstyle{\frac12} v(a_1)+v(x_1),
  \textstyle{\frac12}v(a_2)+v(x_2)\bigr) \qquad\text{for $x_1$,
    $x_2\in F$}
  \]
  is a tame norm compatible with $q$, such that $(\gr_\alpha(U),\tilde
  q_\alpha) \simeq (\grV,\varphi)$ under the map $e_1x_1+e_2x_2\mapsto
  \varepsilon_1\tilde x_1+\varepsilon_2\tilde x_2$. Thus, the Witt
  class of $\varphi$ is in the image of $\partial$.
\end{proof}

Our next goal is to show that $\partial$ is an isomorphism when the
valuation $v$ on $F$ is Henselian.

\begin{lem}
  \label{lem:Hensel}
  Let $(V,q)$ be a quadratic space over $F$. Suppose $v$ is Henselian
  and $x$, $y\in V$ are such that
  \[
  2 v\bigl(b(x,y)\bigr) < v\bigl(q(x)\bigr) + v\bigl(q(y)\bigr).
  \]
  Then there is a $q$-isotropic vector in the span of $x$ and $y$.
\end{lem}

\begin{proof}
  This was observed by several authors, in particular Springer
  \cite[Proposition~1]{S} (see also \cite[Lemma~(2.2)]{Tietze}). We
  give a proof for completeness. First, 
  note that if $z_1$, $z_2\in V$ are 
  multiple of each other, then
   \begin{equation}
    \label{eq:mult}
    2v\bigl(b(z_1,z_2)\bigr) = v\bigl(q(z_1)\bigr) + v\bigl(q(z_2)\bigr) +
    2v(2) \geq v\bigl(q(z_1)\bigr) + v\bigl(q(z_2)\bigr).
  \end{equation}
  In particular, under the hypotheses of the lemma, $x$ and $y$ are
  not multiple of each other. Set 
  $z=yq(x)b(x,y)^{-1}$, so $b(x,z)=q(x)$. For all $\lambda\in F$ we have
  \begin{equation}
    \label{eq:q}
    q(x\lambda +z) = q(x) (\lambda^2+\lambda + q(x)q(y)b(x,y)^{-2}).
  \end{equation}
  The second factor on the right side is a polynomial $P(\lambda)$
  with coefficients in the valuation ring of $F$. Its image in
  $\overline{F}[\lambda]$ is $\lambda(\lambda+1)$. By Hensel's Lemma,
  $P(\lambda)$ has a root $\lambda_0\in F$. The vector $x\lambda_0+z$
  is nonzero since $x$ and $y$ are not multiple of each other, and it
  is an isotropic vector of $q$.
\end{proof}

\begin{thm}
  \label{thm:Wittindex}
  Let $(V,q)$ be a quadratic space over $F$. Assume the valuation $v$
  on $F$ is Henselian. Suppose $\alpha\colon
  V\to\Gamma\cup\{\infty\}$ is a tame norm compatible with $q$, and
  consider the induced quadratic form $\tilde q_\alpha$ on
  $\gr_\alpha(V)$. The forms $q$ and $\tilde q_\alpha$ have the same
  Witt index:
  \[
  \mathfrak{i}_0(q) = \mathfrak{i}_0(\tilde q_\alpha).
  \]
\end{thm}

\begin{proof}
  Lemma~\ref{lem:hyperb} already yields $\mathfrak{i}_0(q) \leq
  \mathfrak{i}_0(\tilde q_\alpha)$. To prove the reverse inequality,
  we argue by induction on $\dim V$, 
  and copy the proof of \cite[Proposition~4.3]{RTW}. If $\tilde
  q_\alpha$ is isotropic, we may find a homogeneous isotropic vector
  $\tilde x$ by taking the homogeneous component of smallest degree of
  an arbitrary isotropic vector. Since $\alpha$ is a tame norm
  compatible with $q$, the polar form $\tilde b_\alpha$ is
  nondegenerate. Therefore, we may find a homogeneous vector $\tilde y$
  such that $\tilde b_\alpha(\tilde x,\tilde y)\neq0$. Let $W\subseteq
  V$ be the subspace spanned by $x$ and $y$, so
  $\gr_\alpha(W)\subseteq \gr_\alpha(V)$ is the graded subspace
  spanned by $\tilde x$ and $\tilde y$. We claim that $W$ is a
  hyperbolic plane. To see this, observe that the equation $\tilde
  q_\alpha(\tilde x)=0$ yields $v\bigl(q(x)\bigr)>2\alpha(x)$, while
  $\tilde b_\alpha(\tilde x,\tilde y)\neq0$ shows that
  $v\bigl(b(x,y)\bigr) =\alpha(x)+\alpha(y)$. Since
  $v\bigl(q(y)\bigr)\geq2\alpha(y)$, it follows that
  $2v\bigl(b(x,y)\bigr) < v\bigl(q(x)\bigr) +
  v\bigl(q(y)\bigr)$. Therefore, Lemma~\ref{lem:Hensel} shows that $W$
  contains a 
  $q$-isotropic vector. The restriction of $q$ to $W$ is nonsingular
  since $b(x,y)\neq0$, hence $W$ is a hyperbolic plane.

  Let $q'$ denote the restriction of $q$ to $W^\perp$, so
  $\mathfrak{i}_0(q) = 1+\mathfrak{i}_0(q')$. By the choice of
  $x$ and $y$, the bilinear form $\tilde b_\alpha$ is nondegenerate on
  $\gr_\alpha(W)$. Therefore, the norm $\alpha\rvert_W$ is compatible
  with the restriction of the polar form $b\rvert_W$, and
  \cite[Proposition~3.8]{RTW} shows that $W^\perp$ is a splitting
  complement of $W$ with respect to $\alpha$, i.e.,
  $\alpha=\alpha\rvert_W\oplus \alpha\rvert_{W^\perp}$. Thus,
  \[
  \gr_\alpha(V) = \gr_\alpha(W) \perp \gr_\alpha(W^\perp),
  \]
  hence $\mathfrak{i}_0(\tilde q_\alpha) = 1+ \mathfrak{i}_0(\tilde
  q'_\alpha)$ since $\gr_\alpha(W)$ is a hyperbolic plane. The
  induction hypothesis yields $\mathfrak{i}_0(q') \geq
  \mathfrak{i}_0(\tilde q'_\alpha)$, hence $\mathfrak{i}_0(q) \geq
  \mathfrak{i}_0(\tilde q_\alpha)$ and the proof is complete.
\end{proof}

\begin{cor}
  \label{cor:deliso}
  The homomorphism $\partial$ of Proposition~\ref{prop:residue} is an
  isomorphism if $F$ is Henselian.
\end{cor}

\begin{proof}
  Surjectivity of $\partial$ was shown in
  Proposition~\ref{prop:residue}, and injectivity follows from
  Theorem~\ref{thm:Wittindex}. 
\end{proof}

\section{Tame quadratic forms over Henselian fields}
\label{sec:tame}

In this section, we show that the tame quadratic Witt group of a
Henselian field is the Witt kernel of the scalar extension map to the
maximal tame extension.

Let $F$ be a field with a valuation $v\colon
F\to\Gamma\cup\{\infty\}$. Throughout this section, we assume $v$ is
Henselian. Let $(V,q)$ be a quadratic space over $F$ with polar form
$b$. If $q$ is anisotropic, we define a map $\hatalpha\colon
V\to\Gamma\cup\{\infty\}$ by
\begin{equation}
\label{eq:alpha}
\hatalpha(x)={\textstyle\frac12}v\bigl(q(x)\bigr)\qquad\text{for $x\in
  V$.}
\end{equation}

\begin{prop}
  \label{prop:properties}
  The map $\hatalpha$ is a $v$-value function and $\hatalpha\prec q$.
\end{prop}

\begin{proof}
  This was proved by Springer
  \cite[Proposition~1]{S} (and attributed to M.~Eichler) in the case
  where the valuation is discrete and $F$ is complete. Springer's
  arguments hold without change under the more general hypotheses of this
  section. We include the proof for the reader's convenience.

  By definition of $\hatalpha$, we clearly have $\hatalpha(x\lambda) =
  \hatalpha(x)+v(\lambda)$ for $x\in V$ and $\lambda\in F$, and
  $\hatalpha(x)=\infty$ if and only if $x=0$ because $q$ is
  anisotropic. It is also clear from the definition that
  $v\bigl(q(x)\bigr)\geq2\hatalpha(x)$ for all $x\in V$. Thus, it
  only remains to show
  \begin{equation}
    \label{eq:iv}
    v\bigl(b(x,y)\bigr) \geq \hatalpha(x)+\hatalpha(y)\qquad\text{for
      $x$, $y\in V$}
  \end{equation}
  and
  \begin{equation}
    \label{eq:iii}
    \hatalpha(x+y)\geq\min\bigl(\hatalpha(x),\hatalpha(y)\bigr)
    \qquad\text{for $x$, $y\in V$.}
  \end{equation}
  The inequality~\eqref{eq:iv} readily follows from
  Lemma~\ref{lem:Hensel} since $q$ is anisotropic.
  Property~\eqref{eq:iii} follows, since \eqref{eq:iv} implies
  $v\bigl(b(x,y)\bigr)\geq \min\bigl(v\bigl(q(x)\bigr),
  v\bigl(q(y)\bigr)\bigr)$ for $x$, $y\in V$, hence
  \[
  v\bigl(q(x+y)\bigr) = v\bigl(q(x)+b(x,y) + q(y)\bigr) \geq
  \min\bigl(v\bigl(q(x)\bigr), v\bigl(q(y)\bigr)\bigr).
  \]
\end{proof}

Now, consider the following property for an anisotropic quadratic form
$q$ over $F$:
\begin{equation}
  \tag{S}
  v\bigl(b(x,y)\bigr) > \hatalpha(x) + \hatalpha(y) \qquad\text{for all
    nonzero $x$, $y\in V$ such that $\hatalpha(x)=\hatalpha(y)$.}
\end{equation}
Note that this property cannot hold if $\charac\overline{F}\neq2$
since for $x=y$ it implies $v(2)>0$ (see
\eqref{eq:mult}). By~\eqref{eq:iv}, we have 
$v\bigl(b(x,y)\bigr)>2\min\bigl(\hatalpha(x),\hatalpha(y)\bigr)$ if
$\hatalpha(x)\neq\hatalpha(y)$, hence Property~(S) has the following
equivalent formulation: for all nonzero $x$, $y\in V$,
\begin{equation}
  \tag{S'}
  v\bigl(b(x,y)\bigr)>2\min\bigl(\hatalpha(x),\hatalpha(y)\bigr) =
  \min\bigl(v\bigl(q(x)\bigr), v\bigl(q(y)\bigr)\bigr).
\end{equation}

Our goal is to prove that Property~(S) characterizes the quadratic
forms that remain anisotropic over any inertial extension. Thus,
\emph{until the end of this section, except for the last theorem, we
  assume $\charac\overline{F}=2$.} 

We first
consider inertial \emph{quadratic} extensions $K/F$. The residue
extension $\overline{K}/\overline{F}$ is obtained by adjoining to
$\overline{F}$ a root of an irreducible polynomial of the form
$X^2+X+\overline{u}$ for some $\overline{u}\in\overline{F}$, hence
$K=F(\mu)$ where $\mu^2+\mu+u=0$ for some $u\in F$ with $v(u)=0$. The
Henselian valuation on $F$ has a unique extension to $K$, for which we
also use the notation $v$. Since $1$, $\overline{\mu}$ are linearly
independent over $\overline{F}$, we have
\begin{equation}
  \label{eq:vmin}
  v(a+b\mu) = \min\bigl(v(a), v(b)\bigr)\qquad\text{for $a$, $b\in F$.}
\end{equation}

\begin{lem}
  \label{lem:quadext}
  An anisotropic quadratic form satisfies~(S) if and only if it
  remains anisotropic over every inertial quadratic extension.
\end{lem}

\begin{proof}
  If (S) does not hold, then we can find nonzero vectors $x$, $y\in V$
  such that $\hatalpha(x)=\hatalpha(y)$ and $v\bigl(b(x,y)\bigr) =
  \hatalpha(x)+\hatalpha(y)$. In view of \eqref{eq:mult}, the vectors $x$
  and $y$ are not multiple of each other. Letting $z=yq(x)b(x,y)^{-1}$
  as in the proof of Lemma~\ref{lem:Hensel}, we see
  from~\eqref{eq:q} that $q$ becomes isotropic over $F(\lambda_0)$,
  where $\lambda_0$ is a root of the polynomial
  $\lambda^2+\lambda+q(x)q(y)b(x,y)^{-2}$. Since
  $v(q(x)q(y)b(x,y)^{-2})=0$, the field $F(\lambda_0)$ is an inertial
  extension of $F$.

  Conversely, suppose $q$ becomes isotropic over an inertial quadratic
  extension $K$ of $F$. We have $K=F(\mu)$ where $\mu^2+\mu+u=0$ for
  some $u\in F$ with $v(u)=0$. Let $x_0\otimes1+ x_1\otimes\mu\in
  V\otimes_FK$ be an isotropic vector of $q$. We have
  \begin{equation}
    \label{eq:qK}
    q(x_0\otimes 1+x_1\otimes\mu) = \bigl(q(x_0)-uq(x_1)\bigr) +
    \bigl(b(x_0,x_1)-q(x_1)\bigr)\mu,
  \end{equation}
  hence 
  \[
  q(x_0)=uq(x_1)\qquad\text{and}\qquad b(x_0,x_1)=q(x_1).
  \]
  Since $v(u)=0$, it follows that $v\bigl(q(x_0)\bigr) =
  v\bigl(q(x_1)\bigr) = v\bigl(b(x_0,x_1)\bigr)$, hence (S) does not hold.
\end{proof}

\begin{lem}
  \label{lem:quadextS}
  If an anisotropic quadratic form satisfies~(S), then its scalar
  extension to any inertial quadratic extension also satisfies~(S).
\end{lem}

\begin{proof}
  Suppose $q$ is an anisotropic quadratic form over $F$
  satisfying~(S), and $K$ is an inertial quadratic extension of
  $F$. By Lemma~\ref{lem:quadext}, we know $q$ remains anisotropic
  over $K$. We may therefore extend to $V\otimes_FK$ the definition of
  the $v$-value function $\hatalpha$ of \eqref{eq:alpha}. Let $K=F(\mu)$
  where $\mu^2+\mu+u=0$ for some $u\in F$ with $v(u)=0$. For
  $x=x_0\otimes1 +x_1\otimes\mu\in V\otimes_FK$ we have by
  \eqref{eq:qK} and \eqref{eq:vmin}
  \[
  \hatalpha(x)={\textstyle\frac12}\min\bigl(v\bigl(q(x_0)-uq(x_1)\bigr),
  v\bigl(b(x_0,x_1)-q(x_1)\bigr)\bigr).
  \]
  We claim that
  \[
  \hatalpha(x) = \min\bigl(\hatalpha(x_0), \hatalpha(x_1)\bigr).
  \]
  We check this formula case-by-case:
  if $v\bigl(q(x_0)\bigr) = v\bigl(q(x_1)\bigr)$, then
  $v\bigl(q(x_0)-uq(x_1)\bigr)\geq v\bigl(q(x_1)\bigr)$ and, by~(S),
  $v\bigl(b(x_0,x_1)-q(x_1)\bigr)=v\bigl(q(x_1)\bigr)$; thus,
  $\hatalpha(x)=\hatalpha(x_0)=\hatalpha(x_1)$.

  If $v\bigl(q(x_0)\bigr)> v\bigl(q(x_1)\bigr)$, then
  $v\bigl(q(x_0)-uq(x_1)\bigr) = v\bigl(q(x_1)\bigr)$ and
  $v\bigl(b(x_0,x_1)\bigr)> v\bigl(q(x_1)\bigr)$ by~(S'), hence
  $\hatalpha(x)=\hatalpha(x_1)$.

  Likewise, if $v\bigl(q(x_0)\bigr)< v\bigl(q(x_1)\bigr)$, then
  $v\bigl(b(x_0,x_1)\bigr)> v\bigl(q(x_0)\bigr)$, hence
  $\hatalpha(x)=\hatalpha(x_0)$. Thus, the claim is proved.

  Now, suppose $x=x_0\otimes 1 + x_1\otimes\mu$ and $y=y_0\otimes1 +
  y_1\otimes \mu$ are nonzero vectors in $V\otimes_FK$ such that
  $\hatalpha(x)=\hatalpha(y)$, and let $\gamma=\hatalpha(x)$, hence
  \[
  \gamma=\min\bigl(\hatalpha(x_0),\hatalpha(x_1)\bigr) =
  \min\bigl(\hatalpha(y_0),\hatalpha(y_1)\bigr).
  \]
  We have
  \[
  b(x,y) = \bigl(b(x_0,y_0) -ub(x_1,y_1)\bigr) + \bigl(b(x_0,y_1) +
  b(x_1,y_0) - b(x_1,y_1)\bigr) \mu,
  \]
  hence by \eqref{eq:vmin}
  \[
  v\bigl(b(x,y)\bigr) = \min\bigl(v\bigl(b(x_0,y_0) -
  ub(x_1,y_1)\bigr), v\bigl(b(x_0,y_1) + b(x_1,y_0) - b(x_1,y_1)\bigr)
  \bigr).
  \]
  By Property~(S') we have 
  \[
  v\bigl(b(x_i,y_j)\bigr) > \min\bigl(\hatalpha(x_i), \hatalpha(y_j)\bigr)
  \geq2\gamma \qquad \text{for $i$, $j=0$, $1$,}
  \]
  hence 
  \[
  v\bigl(b(x_0,y_0) -
  ub(x_1,y_1)\bigr)>2\gamma \quad\text{and}\quad
  v\bigl(b(x_0,y_1) + b(x_1,y_0) - b(x_1,y_1)\bigr)>2\gamma.
  \]
  Therefore, $v\bigl(b(x,y)\bigr)>2\gamma$, and it follows that
  Property~(S) holds for the extension $q_K$ of $q$ to $K$.
\end{proof}

We now turn to odd-degree extensions. Let $L/F$ be an inertial
extension of odd degree~$d$. The Henselian valuation $v$ has a unique
extension to $L$, for which we also use the notation $v$, and we have
$L=F(\lambda)$ for some $\lambda$ with $v(\lambda)=0$ such that
$\overline{L}=\overline{F}(\overline{\lambda})$ and the minimal
polynomial of $\overline{\lambda}$ over $\overline{F}$ has
degree~$d$. Every anisotropic quadratic form $q\colon V\to F$ remains
anisotropic over $L$ by a theorem of Springer (see
\cite[Corollary~18.5]{EKM}), hence we 
may extend the value function $\hatalpha$ of \eqref{eq:alpha} to
$V\otimes_FL$.

\begin{lem}
  \label{lem:oddext}
  For $x=x_0\otimes1+x_1\otimes\lambda+\cdots+ x_{d-1}\otimes
  \lambda^{d-1} \in V\otimes_FL$, we have
  \[
  \hatalpha(x)=\min\bigl(\hatalpha(x_0),\ldots, \hatalpha(x_{d-1})\bigr).
  \]
  If $q$ satisfies~(S), then its extension $q_L$ to $L$ also
  satisfies~(S). 
\end{lem}

\begin{proof}
  If $x=0$, the formula trivially holds. We may therefore assume
  $x_0$, \ldots, $x_{d-1}$ are not all zero and set
  \[
  \gamma=\min\bigl(\hatalpha(x_0),\ldots, \hatalpha(x_{d-1})\bigr).
  \]
  We have
  \[
  q_L(x) = \sum_{0\leq i\leq d-1} q(x_i)\lambda^{2i} + \sum_{0\leq i<j
    \leq d-1} b(x_i,x_j) \lambda^{i+j}.
  \]
  Since $v\bigl(q(x_i)\bigr)\geq 2\gamma$ by definition of $\gamma$
  and
  \[
  v\bigl(b(x_i,x_j)\bigr) \geq \hatalpha(x_i)+\hatalpha(x_j) \geq 2\gamma
  \]
  by \eqref{eq:iv}, it follows that
  $v\bigl(q_L(x)\bigr)\geq2\gamma$, hence $\hatalpha(x)\geq\gamma$.
  To prove that this inequality is an equality, consider the
  homogeneous component $V_\gamma$ of the graded vector space
  $\gr_{\hatalpha}(V)$ and the quadratic map $\tilde q_{\hatalpha}\colon
  V_\gamma\to F_{2\gamma}$ induced by $q$, as in
  \S\ref{sec:norms}. This map is anisotropic since $v\bigl(q(z)\bigr)
  = 2\hatalpha(z)$ for all $z\in V$, by definition of
  $\hatalpha$. Since $\overline{L}/\overline{F}$ is an odd-degree
  extension, this map remains anisotropic after scalar extension to
  $\overline{L}$, by a theorem of Springer (see
  \cite[Corollary~18.5]{EKM}). Therefore, letting
  $x'_i=x_i+V^{>\gamma}\in V_\gamma$ for $i=0$, \ldots, $d-1$, so that
  $x'_i=\tilde x_i\neq 0$ if $\hatalpha(x_i)=\gamma$ and $x'_i=0$ if
  $\hatalpha(x_i)>\gamma$, we have
  \[
  (\tilde q_{\hatalpha})_{\overline{L}}(x'_0\otimes 1 + x'_1\otimes
  \overline{\lambda}+ \cdots+x'_{d-1}\otimes\overline{\lambda}^{d-1})
  \neq 0.
  \]
  The left side is the image of $q_L(x)$ under the canonical map
  $L^{\geq2\gamma}\to L_{2\gamma}$, hence
  $v\bigl(q_L(x)\bigr)=2\gamma$, and $\hatalpha(x)=\gamma$. The
  formula for $\hatalpha(x)$ is thus established. 

  Now, let $x$, $y\in V\otimes_FL$ be nonzero vectors with
  $\hatalpha(x)=\hatalpha(y)$. Let $\gamma=\hatalpha(x)=\hatalpha(y)$ and
  \[
  x=x_0\otimes 1 +\cdots + x_{d-1}\otimes\lambda^{d-1}, \qquad
  y=y_0\otimes 1+\cdots+ y_{d-1}\otimes \lambda^{d-1}
  \]
  with $x_0$, \ldots, $y_{d-1}\in V$ and
  \[
  \gamma=\min\bigl(\hatalpha(x_0), \ldots, \hatalpha(x_{d-1})\bigr) =
  \min\bigl(\hatalpha(y_0), \ldots, \hatalpha(y_{d-1})\bigr).
  \]
  We have
  \[
  b(x,y) = \sum_{i,j=0}^{d-1} b(x_i, y_j)\lambda^{i+j}.
  \]
  If Property~(S) (hence also (S')) holds for $q$, we have
  \[
  v\bigl(b(x_i,y_j)\bigr) > 2\min\bigl(\hatalpha(x_i),\hatalpha(y_j)\bigr)
  \geq 2\gamma \qquad\text{for all $i$, $j$},
  \]
  hence also $v\bigl(b(x,y)\bigr)>2\gamma$. Thus, Property~(S) holds
  for $q_L$.
\end{proof}

\begin{thm}
  \label{thm:inertnonisot}
  Suppose $F$ is a field with a Henselian dyadic valuation. 
  An aniso\-tropic quadratic form over $F$ satisfies~(S) if and only if it
  remains anisotropic over every inertial extension of $F$.
\end{thm}

\begin{proof}
  The ``if'' part follows from Lemma~\ref{lem:quadext}. For the
  converse, suppose $q$ is an anisotropic form over $F$
  satisfying~(S), and let $M$ be an inertial extension of $F$. We have
  to show that $q$ remains anisotropic over $M$. Substituting for $M$
  its Galois closure, we may assume $M$ is Galois over $F$. Let
  $L\subseteq M$ be the subfield fixed under some $2$-Sylow subgroup
  of the Galois group. Then $L/F$ is an odd-degree extension and $M/L$
  is a Galois $2$-extension, hence there is a sequence of field
  extensions
  \[
  L=L_0\subseteq L_1\subseteq \cdots \subseteq L_r = M
  \]
  with $[L_i:L_{i-1}] = 2$ for $i=1$, \ldots, $r$. By a theorem of
  Springer (see \cite[Corollary~18.5]{EKM}), $q$ remains anisotropic
  over $L$, and its 
  extension $q_L$ satisfies~(S) by Lemma~\ref{lem:oddext}. Therefore,
  $q_{L_1}$ is anisotropic by Lemma~\ref{lem:quadext}, and it
  satisfies~(S) by Lemma~\ref{lem:quadextS}. Arguing iteratively, we
  see that $q_M$ is anisotropic (and satisfies~(S)).
\end{proof}

Theorem~\ref{thm:inertnonisot} also yields precise information on the
quadratic forms that are split by an inertial extension of the base
field.

\begin{cor}
  \label{cor:inertsplit}
  Let $q$ be an anisotropic quadratic form over the Henselian dyadic field
  $F$. If $q$ becomes hyperbolic over some inertial extension of $F$, 
  then there exist inertial quadratic extensions $K_1$, \ldots, $K_n$
  of $F$ and $a_1$, \ldots, $a_n\in F^\times$ such that
  \[
  q\simeq \langle a_1\rangle N_1 \perp \ldots \perp \langle a_n\rangle
  N_n
  \]
  where $N_i$ is the (quadratic) norm form of $K_i/F$.
\end{cor}

\begin{proof}
  Since $q$ becomes hyperbolic over some inertial extension of $F$,
  Theorem~\ref{thm:inertnonisot} shows that it does not satisfy
  Property~(S). Therefore, by Lemma~\ref{lem:quadext}, $q$ becomes
  isotropic over some inertial quadratic extension $K_1/F$. By
  \cite[Proposition~34.8]{EKM}, we 
  may then find $a_1\in F^\times$ and a quadratic form $q_1$ such that
  \[
  q\simeq \langle a_1\rangle N_1\perp q_1.
  \]
  Since $q$ becomes hyperbolic over
  some inertial extension of $F$ and $N_1$ becomes hyperbolic over the
  inertial extension $K_1$, the form $q_1$ also becomes hyperbolic
  over some inertial extension of $F$. The corollary follows by
  induction on the dimension.
\end{proof}

If we include the split quadratic $F$-algebra $F\times F$, with
hyperbolic norm form, among the inertial quadratic extensions of $F$,
Witt's decomposition theorem shows that Corollary~\ref{cor:inertsplit}
also holds for isotropic quadratic forms that become
hyperbolic over some inertial extension. It also holds in the
nondyadic case, because then Springer's theorem can be used to show
that every anisotropic quadratic form split by an inertial extension
becomes isotropic over some inertial quadratic extension. Thus, the
same argument as in the proof of Corollary~\ref{cor:inertsplit} applies.

Our last theorem holds without the hypothesis that
$\charac\overline{F}=2$. 

\begin{thm}
  \label{thm:main}
  Let $F$ be a field with a Henselian valuation $v$ (with arbitrary
  residue characteristic), and let $(V,q)$
  be an even-dimensional nondegenerate quadratic space over $F$. The
  following conditions are equivalent:
  \begin{enumerate}
  \item[(a)]
  $q$ becomes hyperbolic over some tame extension of $F$;
  \item[(b)]
  there exists a tame norm $\alpha$ on $V$ compatible with $q$.
  \end{enumerate}
\end{thm}

\begin{proof}
  If $\charac\overline{F}\neq2$, both conditions hold for all
  even-dimensional nondegenerate
  quadratic spaces: (a)~because the quadratic
  closure of $F$ is tame and (b)~by \cite[Corollary~3.6]{RTW}. For
  the rest of the proof, we may thus assume
  $\charac\overline{F}=2$. We may also assume $q$ is anisotropic
  because if (b)~holds for an anisotropic form it also holds for every
  Witt-equivalent form by Lemma~\ref{lem:sum} since hyperbolic forms
  admit tame compatible norms (see Example~\ref{ex:notbounded}).

  If (a) holds, then by Corollary~\ref{cor:inertsplit} we may find a
  decomposition
  \[
  q\simeq\langle a_1\rangle N_1\perp\ldots\perp\langle a_n\rangle N_n
  \]
  for some $a_1$, \ldots, $a_n\in F^\times$ and for $N_1$, \ldots,
  $N_n$ the norm forms of some inertial quadratic extensions $K_1$,
  \ldots, $K_n$ of $F$. Example~\ref{ex:quadext} shows that for each
  $i=1$, \ldots, $n$ the unique valuation on $K_i$ extending $v$ is a
  tame norm compatible with $N_i$, hence also with $\langle a_i\rangle
  N_i$. The direct sum of these norms is a tame norm compatible with
  $q$ by Lemma~\ref{lem:sum}, hence condition~(b) holds.

  Conversely, suppose there is a tame norm $\alpha$ on $V$ compatible
  with $q$. Then we may find a separable extension $L'$ of
  $\overline{F}$ such that $\tilde q_\alpha$ splits after scalar
  extension to $\gr(F)\otimes_{\overline{F}}L'$. Let $L$ be an
  inertial lift of $L'$, i.e., $L/F$ is an inertial extension such
  that $\overline{L}=L'$. Write again $v$ for the unique extension of
  $v$ to $L$. Then $\alpha\otimes v$ is a tame norm on $V\otimes_FL$
  compatible with $q_L$, and $(\tilde{q_L})_{\alpha\otimes v} =
  (\tilde q_\alpha)_{\gr(L)}$ is split since
  $\gr(L)=\gr(F)\otimes_{\overline{F}} L'$. Therefore,
  Theorem~\ref{thm:Wittindex} shows that $q_L$ is hyperbolic.
\end{proof}

As an example, consider the field $F=\Q_2((t))$ of Laurent series in
one indeterminate over the field of $2$-adic numbers. The composite of
the $2$-adic valuation on $\Q_2$ and the $t$-adic valuation on $F$ is
a Henselian valuation $v$ on $F$ with value group $\Z^2$ and residue
field $\F_2$. We have $I_q(\F_2)\simeq\Z/2\Z$, and the unique nontrivial
Witt class is represented by the norm form of $\F_4$. The unique
inertial quadratic extension of $F$ is $F(\sqrt{5})$, and it follows
from Theorem~\ref{thm:gradedWitt}, Corollary~\ref{cor:deliso},
Corollary~\ref{cor:inertsplit} and Theorem~\ref{thm:main} that
\[
I_{qt}(F)\simeq(\Z/2\Z)^4,
\]
with generators the Witt classes of the following forms:
$\langle1,-5\rangle$, 
$\langle2\rangle\langle1,-5\rangle$, $\langle
t\rangle\langle1,-5\rangle$, $\langle2t\rangle\langle1,-5\rangle$. 

By contrast, the full Witt groups $W(F)$, $I_q(F)$ may also be
determined by using Springer's theorem for the $t$-adic valuation,
since the Witt group of $\Q_2$ is known (see for instance
\cite[Ch.~VI, Remark~2.31]{Lam}); we thus get
\[
W(F)\simeq (\Z/8\Z)^2\oplus(\Z/2\Z)^4,
\]
with generators the Witt classes of $\langle1\rangle$, $\langle
t\rangle$, $\langle1,-2\rangle$, $\langle t\rangle\langle1,-2\rangle$,
$\langle1,-5\rangle$, and $\langle t\rangle\langle1,-5\rangle$. Note
that $4\langle1\rangle\simeq\langle1,-2,-5,10\rangle$ over $\Q_2$ (see
\cite[\emph{loc. cit.}]{Lam}), hence
$\langle2\rangle\langle1,-5\rangle$ is Witt-equivalent to
$\langle1,-5\rangle - 4\langle1\rangle$. Therefore, 
\[
I_q(F)/I_{qt}(F)\simeq(\Z/4\Z)\oplus(\Z/2\Z)^3
\]
with generators represented by $\langle1,t\rangle$,
$\langle1,1\rangle$, $\langle1,-2\rangle$, and $\langle
t\rangle\langle1,-2\rangle$.

\end{document}